\documentclass[]{amsart}

\usepackage[english]{babel}
\usepackage[utf8]{inputenc}
\usepackage[T1]{fontenc}
\usepackage{paralist}
\usepackage{amsmath, amsfonts, amssymb, amsthm}
\usepackage{color}
\usepackage[normalem]{ulem}
\usepackage{tikz}

\newcommand{\CC}{\mathbb{C}}

\newcommand{\NN}{\mathbb{N}}

\newcommand{\ZZ}{\mathbb{Z}}

\newcommand{\Oc}{\mathcal{O}}

\newcommand{\nbf}{\mathbf{n}}
\newcommand{\xbf}{\mathbf{x}}
\newcommand{\bsa}{\boldsymbol{\alpha}}
\newcommand{\set}[1]{\left\{ #1 \right\}}
\newcommand{\setb}[1]{\left( #1 \right)}
\newcommand{\abs}[1]{\left| #1 \right|}

\newcommand{\norm}[1]{\left\lVert #1 \right\rVert}

\newtheorem{mymasterthm}{notForUse}
\theoremstyle{definition}

\newtheorem{myrem}[mymasterthm]{Remark}

\theoremstyle{plain}
\newtheorem{mylemma}[mymasterthm]{Lemma}
\newtheorem{mythm}[mymasterthm]{Theorem}

\allowdisplaybreaks

\title[Growth of multi-recurrences]{On the growth of multi-recurrences}
\makeatletter
\@namedef{subjclassname@2020}{\textup{2020} Mathematics Subject Classification}
\makeatother

\author[C. Fuchs]{Clemens Fuchs}
\address{Clemens Fuchs\newline
	\indent University of Salzburg\newline
	\indent Department of Mathematics\newline
	\indent Hellbrunnerstr. 34 \newline
	\indent A-5020 Salzburg, Austria}
\email{clemens.fuchs@plus.ac.at}

\author[S. Heintze]{Sebastian Heintze}
\address{Sebastian Heintze\newline
	\indent Graz University of Technology\newline
	\indent Institute of Analysis and Number Theory\newline
	\indent Steyrergasse 30/II \newline
	\indent A-8010 Graz, Austria}
\email{heintze@math.tugraz.at}

\thanks{Supported by Austrian Science Fund (FWF): I4406.}

\begin{document}
	
	\maketitle
	
	
	\begin{abstract}
		In this paper we provide a complete proof for a bound on the growth of multi-recurrences which are defined over a number field.
		The proven bound was already stated by van der Poorten and Schlickewei forty years ago.
	\end{abstract}
	
	\section{Introduction}
	
	A linear recurrence sequence is a sequence $ (G_n)_{n=0}^{\infty} $ given by a recursive formula of the shape $ G_{n+k} = c_{k-1} G_{n+k-1} + \cdots + c_0 G_n $ together with finitely many initial values.
	This definition makes sense over any field $ K $.
	It is well known that any linear recurrence sequence has an explicit representation of the form
	\begin{equation}
		\label{eq:lrs}
		G_n = b_1(n) \beta_1^n + \cdots + b_r(n) \beta_r ^n
	\end{equation}
	for polynomials $ b_1,\ldots,b_r $ over $ K(\beta_1,\ldots,\beta_r) $ and elements $ \beta_1,\ldots,\beta_r $ which are algebraic over $ K $, the so-called Binet formula.
	It is also well known (see \cite{fuchs-heintze-p3} for references and a proof) that if $ G_n $ takes values in a number field, then under natural and non-restrictive conditions, i.e. $ \beta_1,\ldots,\beta_r $ are algebraic integers, no ratio $ \beta_i / \beta_j $ for $ i \neq j $ is a root of unity and $ \max_{i=1,\ldots,r} \abs{\beta_i} > 1 $, for large enough $ n $ we have
	\begin{equation}
		\label{eq:boundlrscase}
		\abs{G_n} \geq \left( \max_{i=1,\ldots,r} \abs{\beta_i} \right)^{n(1-\varepsilon)}
	\end{equation}
	for $ \varepsilon > 0 $.
	An analogous result holds true if $ G_n $ takes values in a function field in one variable over $ \CC $ as is shown in \cite{fuchs-heintze-p3}.
	
	There is a natural generalization of linear recurrence sequences.
	If we allow more than one parameter, we can generalize \eqref{eq:lrs} to
	\begin{equation*}
		G(n_1,\ldots,n_s) = \sum_{i=1}^{k} f_i(n_1,\ldots,n_s) \alpha_{i1}^{n_1} \cdots \alpha_{is}^{n_s}
	\end{equation*}
	where $ s $ and $ k $ are positive integers, $ f_1,\ldots,f_k $ are polynomials in $ s $ variables and $ n_1,\ldots,n_s $ are non-negative integers.
	Such polynomial-exponential functions $ G : \NN_0^s \to K $ are called multi-recurrences, where we have denoted the set of non-negative integers by $ \NN_0 $.
	We say that $ G $ is defined over a field $ K $ if the coefficients and the bases $ \alpha_{i1}, \ldots, \alpha_{is} $ for $ i = 1,\ldots,k $ are in $ K $. If $ G $ is defined over $ K $, then it takes values in $ K $.
	For more information about recurrence sequences we refer to \cite{schmidt-2003}.
	Van der Poorten and Schlickewei claimed in \cite{vanderpoorten-schlickewei-1982} a similar bound as \eqref{eq:boundlrscase} for multi-recurrences defined over number fields.
	The purpose of the present paper is to provide a proof for that bound.
	We will do this in the same way as in \cite{fuchs-heintze-p3} for the case of linear recurrence sequences and use the same auxiliary result due to Evertse \cite{evertse-1984}, which is cited as Theorem \ref{thm:valuationprodineq} below.
	
	\section{Notation and result}
	
	In the sequel we shall use the abbreviation
	\begin{equation*}
		f_i(\nbf) \bsa_i^{\nbf}
	\end{equation*}
	for the expression
	\begin{equation*}
		f_i(n_1,\ldots,n_s) \alpha_{i1}^{n_1} \cdots \alpha_{is}^{n_s},
	\end{equation*}
	i.e. we will indicate by boldface letters that the considered object is a vector in difference to a single number.
	Moreover, for a vector $ \nbf \in \ZZ^s $ we consider its norm
	\begin{equation*}
		\abs{\nbf} = \abs{n_1} + \cdots + \abs{n_s}.
	\end{equation*}
	In what follows we are interested in multi-recurrences $ G $ as defined above.
	Our main result is the following theorem:
	
	\begin{mythm}
		\label{thm:multibound}
		Let $ K $ be a number field and $ s $ a positive integer.
		Consider the polynomial-exponential function
		\begin{equation*}
			G(\nbf) = \sum_{i=1}^{k} f_i(\nbf) \bsa_i^{\nbf}
		\end{equation*}
		with non-zero algebraic integers $ \alpha_{ij} \in K $, for $ i=1,\ldots,k $ and $ j=1,\ldots,s $, and polynomials $ f(X_1,\ldots,X_s) \in K[X_1,\ldots,X_s] $.
		Fix $ \varepsilon > 0 $.
		Assume that there is an index $ i_0 $, $ 1 \leq i_0 \leq k $, such that there is no subset $ I \subseteq \set{1,\ldots,k} $ with $ i_0 \in I $ and
		\begin{equation*}
			\sum_{i \in I} f_i(\nbf) \bsa_i^{\nbf} = 0.
		\end{equation*}
		Then, for $ \abs{\nbf} $ large enough we have
		\begin{equation*}
			\abs{G(\nbf)} \geq \abs{f_{i_0}(\nbf) \bsa_{i_0}^{\nbf}} e^{-\varepsilon \abs{\nbf}}.
		\end{equation*}
	\end{mythm}
	
	\begin{myrem}
		The condition concerning $ i_0 $ in the above theorem is really necessary and already stated in \cite{vanderpoorten-schlickewei-1982}. Indeed, the size of $ G(\nbf) $ cannot be bounded by a term from a vanishing subsum.
	\end{myrem}
	
	\begin{myrem}
		\label{rem:generalized}
		We emphasize that the same statement as in Theorem \ref{thm:multibound} holds with the completely analogous proof also for any other valution $ \abs{\cdot}_{\mu} $ on $ K $ in the proven lower bound instead of the standard absolute value.
	\end{myrem}
	
	\begin{myrem}
		The bound in Theorem \ref{thm:multibound} holds for all $ \nbf $ with $ \abs{\nbf} \geq B $.
		Unfortunately this lower bound $ B $ cannot be given explicitely since it depends, among others, on the ineffective constant given by Theorem \ref{thm:valuationprodineq} below. More precisely, it is influenced by a threshold where the exponential function becomes larger than a polynomial function having ineffective coefficients.
	\end{myrem}
	
	We are only able to prove the result for number fields.
	To the knowledge of the authors it is still open to find and prove an analogous result in the function field case, i.e. a version of Theorem 2.1 in \cite{fuchs-heintze-p3} for multi-recurrences.
	We leave this as an open question.
	
	\section{Preliminaries}
	
	In our proof we will need the following result of Evertse. The reader will find it as Theorem 2 in \cite{evertse-1984}. We use the notation
	\begin{equation*}
		\norm{\xbf} := \max_{\substack{k=0,\ldots,t \\ i=1,\ldots,D}} \abs{\sigma_i(x_k)}
	\end{equation*}
	with $ \set{\sigma_1, \ldots, \sigma_D} $ the set of all embedings of $ K $ in $ \CC $ and $ \xbf = (x_0,x_1,\ldots,x_t) $.
	Moreover, we denote by $ \Oc_K $ the ring of integers in the number field $ K $ and by $ M_K $ the set of places of the number field $ K $:
	\begin{mythm}
		\label{thm:valuationprodineq}
		Let $ t $ be a non-negative integer and $ S $ a finite set of places in $ K $, containing all infinite places.
		Then for every $ \varepsilon > 0 $ a constant $ C $ exists, depending only on $ \varepsilon, S, K, t $ such that for each non-empty subset $ T $ of $ S $ and every vector $ \xbf = (x_0,x_1,\ldots,x_t) \in \Oc_K^{t+1} $ with
		\begin{equation*}
			x_{i_0} + x_{i_1} + \cdots + x_{i_s} \neq 0
		\end{equation*}
		for each non-empty subset $ \set{i_0,i_1,\ldots,i_s} $ of $ \set{0,1,\ldots,t} $ the inequality
		\begin{equation*}
			\left( \prod_{k=0}^{t} \prod_{\nu \in S} \abs{x_k}_{\nu} \right) \prod_{\nu \in T} \abs{x_0 + x_1 + \cdots + x_t}_{\nu} \geq C \left( \prod_{\nu \in T} \max_{k=0,\ldots,t} \abs{x_k}_{\nu} \right) \norm{\xbf}^{-\varepsilon}.
		\end{equation*}
		is valid.
	\end{mythm}
	
	Moreover, the next lemma is used in the proof of our theorem.
	It is an analogous version of Lemma 1 in \cite{evertse-1984} for vectors.
	\begin{mylemma}
		\label{lem:boundprod}
		Let $ K $ be a number field of degree $ D $, let $ f(X_1,\ldots,X_s) \in K[X_1,\ldots,X_s] $ be a polynomial of absolute degree $ m $ and let $ T $ be a non-empty set of places in $ K $.
		Then there exists a positive constant $ c $, depending only on $ K,f $, such that for all $ \nbf \in \ZZ^s $ with $ \nbf \neq 0 $ and $ f(\nbf) \neq 0 $ it holds that
		\begin{equation*}
			\prod_{\nu \in T} \abs{f(\nbf)}_{\nu} \leq \prod_{\nu \in M_K} \max \setb{1, \abs{f(\nbf)}_{\nu}} \leq c \abs{\nbf}^{Dm}.
		\end{equation*}
	\end{mylemma}
	
	\begin{proof}
		Obviously we have $ \abs{f(\nbf)}_{\nu} \leq \max \setb{1, \abs{f(\nbf)}_{\nu}} $.
		Thus the first inequality
		\begin{equation*}
			\prod_{\nu \in T} \abs{f(\nbf)}_{\nu} \leq \prod_{\nu \in M_K} \max \setb{1, \abs{f(\nbf)}_{\nu}}
		\end{equation*}
		is trivial.
		Note that for each $ \nbf $ there are only finitely many places $ \nu $ with $ \abs{f(\nbf)}_{\nu} \neq 1 $ and therefore the products are finite.
		
		There are at most $ D $ infinite places. Hence there is a positive constant $ c_1 $ such that
		\begin{equation*}
			\abs{f(\nbf)}_{\nu} \leq c_1 \abs{\nbf}^m
		\end{equation*}
		holds for all infinite places $ \nu $.
		Moreover, for all but finitely many finite places $ \nu $ we have
		\begin{equation*}
			\abs{f(\nbf)}_{\nu} \leq 1.
		\end{equation*}
		These finitely many places depend only on (the denominators of) the coefficients of $ f $ and are independent of $ \nbf $.
		For the finitely many remaining finite places $ \nu $ there is a positive constant $ c_2 $, independent of $ \nbf $, such that
		\begin{equation*}
			\abs{f(\nbf)}_{\nu} \leq c_2.
		\end{equation*}
		We may assume that $ c_1 > 1 $ and $ c_2 > 1 $.
		Putting things together we get
		\begin{equation*}
			\prod_{\nu \in M_K} \max \setb{1, \abs{f(\nbf)}_{\nu}} \leq c \abs{\nbf}^{Dm}
		\end{equation*}
		for a new constant $ c $.
	\end{proof}
	
	\section{Proof of Theorem \ref{thm:multibound}}
	
	Before we start with proving Theorem \ref{thm:multibound} let us mention that the proof follows the same strategy and is the multi-recurrence version of the proof of the corresponding result for linear recurrences given in the appendix of \cite{fuchs-heintze-p3} by the authors.
	
	\begin{proof}[Proof of Theorem \ref{thm:multibound}]
		Since the bases $ \alpha_{ij} $ of the exponential parts of $ G $ are algebraic integers, we can find a non-zero integer $ z $ such that $ z f_i(\nbf) \bsa_i^{\nbf} $ are algebraic integers for all $ i=1,\ldots,k $ and all non-negative integers $ n_1,\ldots,n_s $.
		Choose $ S $ as a finite set of places in $ K $ containing all infinite places as well as all places such that $ \alpha_{ij} $ for $ i=1,\ldots,k $ and $ j=1,\ldots,s $ are $ S $-units.
		Let $ \mu $ be such that $ \abs{\cdot}_{\mu} = \abs{\cdot} $ is the usual absolute value on $ \CC $. In particular we have $ \mu \in S $. Further define $ T = \set{\mu} $.
		
		We may assume that for the index $ i_0 $ from the theorem we have $ i_0 = 1 $ to simplify the notation.
		By renumbering summands we can assume that
		\begin{equation*}
			G(\nbf) = \sum_{i=1}^{\ell} f_i(\nbf) \bsa_i^{\nbf}
		\end{equation*}
		for an integer $ \ell $ with $ 1 \leq \ell \leq k $ has no vanishing subsum.
		Indeed, there are only finitely many possible subsums with this property and we can perform the following steps for each of them analogously. At the end of the proof we can put the cases together by choosing the largest occuring bound for $ \abs{\nbf} $.
		
		Thus we can apply Theorem \ref{thm:valuationprodineq} and get
		\begin{equation*}
			\left( \prod_{i=1}^{\ell} \prod_{\nu \in S} \abs{z f_i(\nbf) \bsa_i^{\nbf}}_{\nu} \right) \abs{zG(\nbf)} \geq C \max_{i=1,\ldots,\ell} \abs{z f_i(\nbf) \bsa_i^{\nbf}} \norm{z\xbf}^{-\varepsilon'}
		\end{equation*}
		for $ \xbf = \left( f_1(\nbf) \bsa_1^{\nbf}, \ldots, f_{\ell}(\nbf) \bsa_{\ell}^{\nbf} \right) $ and an $ \varepsilon' $ to be fixed later.
		Using that $ z $ is a fixed integer and that the $ \alpha_{ij} $ are $ S $-units, we get
		\begin{equation}
			\label{eq:firstsimplified}
			\left( \prod_{i=1}^{\ell} \prod_{\nu \in S} \abs{f_i(\nbf)}_{\nu} \right) \abs{G(\nbf)} \geq C_1 \abs{f_1(\nbf) \bsa_1^{\nbf}} \norm{\xbf}^{-\varepsilon'}.
		\end{equation}
		
		Let $ m $ denote the maximum of the absolute degrees of the polynomials $ f_1,\ldots,f_k $.
		Then there exists a constant $ A > 1 $, which is independent of $ \nbf $ and $ \varepsilon' $ (and fits for all of the finitely many cases mentioned in the second paragraph of this proof), satisfying
		\begin{align*}
			\norm{\xbf} &= \max_{\substack{i=1,\ldots,\ell \\ t=1,\ldots,D}} \abs{\sigma_t \left( f_i(\nbf) \bsa_i^{\nbf} \right)} \\
			&\leq \max_{\substack{i=1,\ldots,\ell \\ t=1,\ldots,D}} \abs{\sigma_t \left( f_i(\nbf) \right)} \cdot \max_{\substack{i=1,\ldots,\ell \\ t=1,\ldots,D}} \abs{\sigma_t \left( \bsa_i^{\nbf} \right)} \\
			&\leq C_2 \abs{\nbf}^m \prod_{j=1}^{s} \max_{\substack{i=1,\ldots,\ell \\ t=1,\ldots,D}} \abs{\sigma_t \left( \alpha_{ij}^{n_j} \right)} \\
			&\leq C_2 \abs{\nbf}^m A^{\abs{\nbf}}.
		\end{align*}
		Inserting this upper bound into inequality \eqref{eq:firstsimplified} yields
		\begin{equation}
			\label{eq:secondsimplified}
			\left( \prod_{i=1}^{\ell} \prod_{\nu \in S} \abs{f_i(\nbf)}_{\nu} \right) \abs{G(\nbf)} \geq \abs{f_1(\nbf) \bsa_1^{\nbf}} C_3 \abs{\nbf}^{-m\varepsilon'} A^{-\abs{\nbf} \varepsilon'}.
		\end{equation}
		
		Now we apply Lemma \ref{lem:boundprod} to the double product in the last displayed inequality.
		This gives us
		\begin{equation*}
			\prod_{i=1}^{\ell} \prod_{\nu \in S} \abs{f_i(\nbf)}_{\nu} \leq \prod_{i=1}^{\ell} C_4^{(i)} \abs{\nbf}^{Dm} \leq C_5 \abs{\nbf}^{Dm\ell}
		\end{equation*}
		and together with inequality \eqref{eq:secondsimplified} the lower bound
		\begin{align*}
			\abs{G(\nbf)} &\geq \abs{f_1(\nbf) \bsa_1^{\nbf}} C_6 \abs{\nbf}^{-Dm\ell -m\varepsilon'} A^{-\abs{\nbf} \varepsilon'} \\
			&\geq \abs{f_1(\nbf) \bsa_1^{\nbf}} A^{-2\varepsilon' \abs{\nbf}}
		\end{align*}
		where the last inequality holds for $ \abs{\nbf} $ large enough.
		Thus, choosing $ \varepsilon' $ such that $ 2\varepsilon' \log(A) = \varepsilon $, we end up with
		\begin{equation*}
			\abs{G(\nbf)} \geq \abs{f_1(\nbf) \bsa_1^{\nbf}} e^{-2\varepsilon' \log(A) \abs{\nbf}}
			= \abs{f_1(\nbf) \bsa_1^{\nbf}} e^{-\varepsilon \abs{\nbf}}
		\end{equation*}
		and the theorem is proven.
	\end{proof}

\end{document}